\documentclass[12pt]{amsart}
\usepackage{amsmath,amssymb,amsfonts}
\usepackage{hyperref}
\usepackage{verbatim}
\usepackage{cancel}
\usepackage{url}
\usepackage{graphicx}

\allowdisplaybreaks
\newtheorem{theorem}{Theorem}

\newtheorem{corollary}[theorem]{Corollary}

\newtheorem{lemma}[theorem]{Lemma}

\newtheorem{proposition}[theorem]{Proposition}

\numberwithin{equation}{section}
\numberwithin{theorem}{section}

\begin{document}
\title{An Oriented Hypergraphic Approach to Algebraic Graph Theory}
\author{Nathan Reff}
\address{Department of Mathematical Sciences\\
Binghamton University (SUNY)\\
Binghamton, NY 13902, U.S.A.}
\email{reff@math.binghamton.edu}
\author{Lucas J. Rusnak}
\address{Department of Mathematics\\
Texas State University\\
San Marcos, TX 78666, USA}
\email{Lucas.Rusnak@txstate.edu}

\begin{abstract}
An oriented hypergraph is a hypergraph where each vertex-edge incidence is given a label of $+1$ or $-1$.  We define the adjacency, incidence and Laplacian matrices of an oriented hypergraph and study each of them.  We extend several matrix results known for graphs and signed graphs to oriented hypergraphs.  New matrix results that are not direct generalizations are also presented.  Finally, we study a new family of matrices that contains walk information.
\end{abstract}

\date{\today}
\maketitle

\section{Introduction}
Researchers have studied the adjacency, Laplacian, normalized Laplacian and signless Laplacian matrices of a graph \cite{MR2571608,AGT}. Recently there has been a growing study of matrices associated to a signed graph \cite{MITTOSSG}.  In this paper we hope to set the foundation for a study of matrices associated with an oriented hypergraph.

An \emph{oriented hypergraph} is a hypergraph with the additional structure that each vertex-edge incidence is given a label of $+1$ or $-1$ \cite{OHD,Shi1}.  A consequence of studying oriented hypergraphs is that graphs and signed graphs can be viewed as specializations.  A graph can be thought of as an oriented hypergraph where each edge is contained in two incidences, and exactly one incidence of each edge is signed $+1$.  A \emph{signed graph} can be thought of as an oriented hypergraph where each edge is contained in two incidences.

We define the adjacency, incidence, and Laplacian matrices of an oriented hypergraph and examine walk counting to extend classical matrix relationships of graphs and signed graphs to the setting of oriented hypergraphs.  New relationships involving the dual hypergraphic structure are also presented, including a result on line graphs of signed graphs.   Finally, we study a new matrix which encapsulates walk information on an oriented hypergraph to provide a unified combinatorial interpretation of the Laplacian matrix entries.

\section{Background}\label{BackgroundSection}
\subsection{Oriented Hypergraphs}

Throughout, $V$ and $E$ will denote disjoint finite sets whose respective
elements are called \emph{vertices} and \emph{edges}. \ An \emph{incidence
function} is a function $\iota :V\times E\rightarrow \mathbb{Z}_{\geq 0}$.  A vertex $v$ and an edge $e$ are said to be \emph{incident} (with
respect to $\iota $) if $\iota (v,e)\neq 0$. \ An \emph{incidence} is a
triple $(v,e,k)$, where $v$ and $e$ are incident and $k\in \{1$, $2$, $3$,\ldots, $\iota (v,e)\}$; the value $\iota (v,e)$ is called the \emph{%
multiplicity} of the incidence.

Let $\mathcal{I}$ be the set of incidences determined by $\iota $. \ An 
\emph{incidence orientation} is a function $\sigma :\mathcal{I}\rightarrow
\{+1,-1\}$.  An \emph{oriented incidence} is a quadruple $(v,e,k,\sigma (v,e,k))$.  An \emph{oriented hypergraph} is a quadruple $(V,E,\mathcal{I},\sigma )$, and
its \emph{underlying hypergraph} is the triple $(V,E,\mathcal{I})$.

A hypergraph is \emph{simple} if $\iota (v,e)\leq 1$ for all $v$ and $e$,
and for convenience we will write $(v,e)$ instead of $(v,e,1)$ if $G$ is a
simple hypergraph. Two, not necessarily distinct, vertices $v$ and $w$ are
said to be \emph{adjacent with respect to edge }$e$ if there exist
incidences $(v,e,k_{1})$ and $(w,e,k_{2})$ such that $(v,e,k_{1})\neq
(w,e,k_{2})$.  An \emph{adjacency} is a quintuple $(v,k_1;w,k_2;e)$, where $v$ and $w$ are adjacent with respect to $e$ using incidences $(v,e,k_1)$ and $(w,e,k_2)$. The \emph{degree} of a vertex $v$, denoted by $\deg (v)$, is equal to the number of incidences containing $v$.  A \emph{$k$-regular hypergraph} is a hypergraph where every vertex has degree $k$.   The \emph{size} of an edge is the number of incidences containing that edge.  A $k$\emph{-edge} is an edge of size $k$.  A \emph{$k$-uniform hypergraph} is a hypergraph where all of its edges have size $k$.

A \emph{walk} is a sequence $W=a_{0},i_{1},a_{1},i_{2},a_{2},i_{3},a_{3},...,a_{n-1},i_{n},a_{n}$ of vertices, edges and incidences, where $\{a_{k}\}$ is an alternating sequence of vertices and edges, $i_{h}$ is an incidence containing $a_{h-1}$ and $a_{h}$, and $i_{2h-1}\neq i_{2h}$. The first and last elements of this sequence are called \emph{anchors}.  A walk where both anchors are vertices is called a \emph{vertex-walk}.  A walk where both
anchors are edges is called an \emph{edge-walk}.  A walk where one anchor
is a vertex and the other is an edge is called a \emph{cross-walk}.  The \emph{length of a walk} is half the number of incidences that appear in that walk. Let $a_{i},a_{j}\in V\cup E$ and $k\in \frac{1}{2}\mathbb{Z}$, define $w(a_{i},a_{j};k)$ to be the number of walks of length $k$ from $a_{i}$ to $a_{j}$. Observe that vertex-walks and edge-walks have integral length, while cross-walks have half-integral length.

Given a hypergraph $G$ the \emph{incidence dual} $G^{\ast }$ is the
hypergraph obtained by reversing the roles of the vertices and edges. Similarly, for an oriented hypergraph $G=(V,E,\mathcal{I},\sigma )$, the \emph{incidence dual} $G^{\ast }$ is the oriented hypergraph $(E,V,\mathcal{I}^{\ast },\sigma ^{\ast })$, where $\mathcal{I}^{\ast }=\{(e,v,k):(v,e,k)\in \mathcal{I}\}$, and $\sigma ^{\ast }:\mathcal{I}^{\ast }\rightarrow \{+1,-1\}$ such that $\sigma ^{\ast }(e,v,k)=\sigma(v,e,k)$.  The set $\mathcal{I}^{\ast }$ determines an incidence
function $\iota ^{\ast }$ where $\iota ^{\ast }(e,v)=\iota (v,e)$.

\begin{proposition}
\label{P2}If $G$ is an oriented hypergraph, then $G^{\ast \ast }=G$.
\end{proposition}

The concept of signed adjacency was introduced in \cite{OHD} as a
hypergraphic alternative to signed edges.  The \emph{sign of the adjacency $(v,k_1;w,k_2;e)$} is defined as 
\begin{equation*}
sgn_{e}(v,k_{1};w,k_{2})=-\sigma (v,e,k_{1})\sigma (w,e,k_{2})\text{.}
\end{equation*}%
This will oftentimes be shortened to $sgn_{e}(v,w)=-\sigma (v,e)\sigma
(w,e)$ when the adjacency is understood, as in the case of simple oriented
hypergraphs.  We will assume $sgn_{e}(v,k_{1};w,k_{2})=0$ if $v$ and $w$ are not adjacent.

The dual to signed adjacency is the \emph{sign of the co-adjacency} 
\begin{equation*}
sgn_{v}(e,k_{1};f,k_{2})=-\sigma (v,e,k_{1})\sigma (v,f,k_{2}).
\end{equation*}

The \emph{sign of a walk} is the product of the signs of all adjacencies in
the walk if it is a vertex-walk, the product of the signs of all
co-adjacencies if it is an edge-walk, and is the product of the signs of all
adjacencies in the walk along with the extra incidence sign if it is a
cross-walk. \ That is, the sign of a walk $%
W=a_{0},i_{1},a_{1},i_{2},a_{2},i_{3},a_{3},...,a_{n-1},i_{n},a_{n}$ is%
\begin{equation}\label{SignOfAWalk}
sgn(W)=(-1)^{p}\prod_{h=1}^{n}\sigma (i_{h})\text{,}
\end{equation}
where $ p=\lfloor n/2\rfloor$.

There are a number of additional signed walk counts we need besides $%
w(a_{i},a_{j};k)$. Let $w^{+}(a_{i},a_{j};k)$ be the number of
positive walks of length $k$ between anchors $a_{i}$ and $a_{j}$, and $%
w^{-}(a_{i},a_{j};k)$ be the number of negative walks of length $k$ between
anchors $a_{i}$ and $a_{j}$. Finally, let $w^{\pm
}(a_{i},a_{j};k)=w^{+}(a_{i},a_{j};k)-w^{-}(a_{i},a_{j};k)$.

A \emph{vertex-switching function} is any function $\theta:V\rightarrow \{-1,+1\}$.  Vertex-switching the oriented hypergraph $G$ means replacing $\sigma$ by $\sigma^{\theta}$, defined by: $\sigma^{\theta}(v,e,k_1)=\theta(v)\sigma(v,e,k_1)$; producing the oriented hypergraph $G^{\theta}=(V,E,\mathcal{I},\sigma^{\theta})$.  The vetex-switching produces an adjacency signature $sgn^{\theta}$, defined by: $sgn_{e}^{\theta}(v,k_{1};w,k_{2})= \theta(v)sgn_{e}(v,k_{1};w,k_{2})\theta(w)$. 

%======================================================================================================
\subsection{Signed Graphs}

A 2-uniform simple oriented hypergraph has been called a bidirected graph,
and was first studied by Edmonds and Johnson \cite{MR0267898}. \ Since each 2-edge forms
a unique adjacency we can regard the adjacency sign as the sign of the
2-edge, and a bidirection as an orientation of the signed edges. \ Zaslavsky
introduced this concept to signed graphs in \cite{OSG}.

A \emph{signed graph} is a graph with the additional structure that each
edge is given a sign of either $+1$ or $-1$. Formally, a signed graph is a
pair $\Sigma =(\Gamma ,sgn)$ consisting of an \emph{underlying graph} $%
\Gamma =(V,E)$ and a \emph{signature} $sgn:E\rightarrow \{+1,-1\}$. An
unsigned graph can be thought of as a signed graph with all edges signed $+1$%
.  The \emph{sign of an edge $e$} is denoted $sgn(e)$.

An \emph{oriented signed graph} is a pair $%
(\Sigma ,\tau )$, consisting of a signed graph $\Sigma $ and an \emph{%
orientation} $\tau :V\times E\rightarrow \{-1,0,+1\}$, where%
\begin{align*}
\tau (v,e)\tau (w,e)& =-sgn(e)\text{ if }v\text{ and }w\text{ are adjacent
via }e, \\
\tau (v,e)& =0\text{ if }v\text{ and }e\text{ are not incident};
\end{align*}%
furthermore $\tau(v,e)\neq 0$ if $v$ is incident to $e$.  By convention, $\tau (v,e)=+1$ is thought as an arrow pointing toward the
vertex $v$, and $\tau (v,e)=-1$ is thought as an arrow pointing away from
the vertex $v$. Thus, an orientation can be viewed as a bidirection on a signed graph.

\begin{proposition}
\label{SGisaOHG}$G$ is a 2-uniform oriented hypergraph if, and only if, $G$
is an oriented signed graph.
\end{proposition}

The line graph of a signed graph can be defined via oriented signed graphs. The
following definition was introduced by Zaslavsky \cite{MITTOSSG}.  Let $\Lambda (\Gamma )$ denote the line graph of the unsigned graph $\Gamma$. A \emph{line graph of an oriented signed graph $(\Sigma ,\tau )$} is the oriented signed graph $(\Lambda (\Gamma ),\tau _{\Lambda })$, where $\tau _{\Lambda }$ is defined by 
\begin{equation}\label{SGLGorientRel}
\tau _{\Lambda }(e_{ij},e_{ij}e_{jk})=\tau (v_{j},e_{ij}).
\end{equation}%
The signature of this line graph is denoted by $sgn_{\Lambda(\tau)}$.

%========================================================
%========================================================
\section{Adjacency matrix}\label{AdjacencySection}
%========================================================
%========================================================

The \emph{adjacency matrix} $A_{G}=[a_{ij}]$ of a simple oriented hypergraph $G$ is  defined by
\begin{equation*}
a_{ij}=\sum_{e\in E}sgn_{e}(v_{i},v_{j})\text{.}
\end{equation*}

Vertex switching an oriented hypergraph $G$ can be described as matrix conjugation of the adjacency matrix.  For a vertex switching function $\theta$, we define a diagonal matrix $D(\theta):=\text{diag}(\theta(v_i):v_i\in V)$.

\begin{proposition} Let $G$ be an oriented hypergraph.  If $\theta$ is a vertex switching function on $G$, then $A_{G^{\theta}}=D(\theta)^T A_{G} D(\theta)$.
\end{proposition}
%We have an immediate extension of Theorem \ref{G1} to simple oriented hypergraphs.

If $A_{\Gamma }$ is the adjacency matrix of an unsigned graph $\Gamma $, it
is well known that the $(i,j)$-entry of $A_{\Gamma }^{k}$ counts the number
of walks of length $k$ from $v_{i}$ to $v_{j}$.  Zaslavsky generalized this result to signed graphs \cite{MITTOSSG}. Here we present a generalization to oriented hypergraphs.
%-------------------------

\begin{theorem}
\label{T1} If $G$ is a simple oriented hypergraph and $k$ a non-negative
integer, then the $(i,j)$-entry of $A_{G}^{k}$ is $w^{\pm }(v_{i},v_{j};k)$.
\end{theorem}

\begin{proof}
We prove this using mathematical induction.

\textit{Base case:} If $k=0$, then $A_{G}^{0}=I$, which is consistent with a 
$0$-walk travelling nowhere.

If $k=1$, then $\left( A_{G}^{1}\right) _{ij}=a_{ij}$, where%
\begin{equation*}
a_{ij} = \sum_{e\in E}sgn_{e}(v_{i},v_{j}) =
w^{+}(v_{i},v_{j};1)-w^{-}(v_{i},v_{j};1) = w^{\pm }(v_{i},v_{j};1).
\end{equation*}

\textit{Induction hypothesis:} Suppose that $\left( A_{G}^{k}\right)
_{ij}=w^{\pm }(v_{i},v_{j};k)$.

We calculate the $(i,j)$-entry of $A_{G}^{k+1}$ as follows:
\begin{align*}
\left( A_{G}^{k+1}\right) _{ij}=\left( A_{G}A_{G}^{k}\right) _{ij}& =\sum_{l=1}^{n}a_{il}\cdot w^{\pm }(v_{l},v_{j};k)\qquad \text{(by the
Induction Hypothesis)} \\
& =\sum_{l=1}^{n}\left( \sum_{e\in E}sgn_{e}(v_{i},v_{l})\right)
[w^{+}(v_{l},v_{j};k)-w^{-}(v_{l},v_{j};k)] \\
&
=\sum_{l=1}^{n}(w^{+}(v_{i},v_{l};1)-w^{-}(v_{i},v_{l};1))[w^{+}(v_{l},v_{j};k)-w^{-}(v_{l},v_{j};k)]
\\
&
=\sum_{l=1}^{n}[w^{+}(v_{i},v_{l};1)w^{+}(v_{l},v_{j};k)+w^{-}(v_{i},v_{l};1)w^{-}(v_{l},v_{j};k)
\\
& \qquad
-w^{-}(v_{i},v_{l};1)w^{+}(v_{l},v_{j};k)-w^{+}(v_{i},v_{l};1)w^{-}(v_{l},v_{j};k)].
\end{align*}
Notice that the number of positive walks of length $k+1$ from $v_{i}$ to $%
v_{j}$ with $v_{l}$ as the second vertex is 
\begin{equation*}
w^{+}(v_{i},v_{l};1)w^{+}(v_{l},v_{j};k)+w^{-}(v_{i},v_{l};1)w^{-}(v_{l},v_{j};k).
\end{equation*}
Similarly, the number of negative walks of length $k+1$ from $v_{i}$ to $%
v_{j}$ with $v_{l}$ as the second vertex is%
\begin{equation*}
w^{-}(v_{i},v_{l};1)w^{+}(v_{l},v_{j};k)+w^{+}(v_{i},v_{l};1)w^{-}(v_{l},v_{j};k)%
\text{.}
\end{equation*}
Since any walk of length $k+1$ from $v_{i}$ to $v_{j}$ must have one of the
vertices $v_{1},v_{2},\ldots ,v_{n}$ as its second vertex, it must be that%
\begin{equation*}
w^{+}(v_{i},v_{j};k+1)=\sum_{l=1}^{n}\big[ %
w^{+}(v_{i},v_{l};1)w^{+}(v_{l},v_{j};k)+w^{-}(v_{i},v_{l};1)w^{-}(v_{l},v_{j};k)%
\big],
\end{equation*}%
and%
\begin{equation*}
w^{-}(v_{i},v_{j};k+1)=\sum_{l=1}^{n}\big[ %
w^{-}(v_{i},v_{l};1)w^{+}(v_{l},v_{j};k)+w^{+}(v_{i},v_{l};1)w^{-}(v_{l},v_{j};k)%
\big] \text{.}
\end{equation*}
Thus, the $(i,j)$-entry of $A_{G}^{k+1}$ simplifies to%
\begin{equation*}
w^{+}(v_{i},v_{j};k+1)-w^{-}(v_{i},v_{j};k+1)=w^{\pm }(v_{i},v_{j};k+1).\qedhere
\end{equation*}
\end{proof}

Similarly, using the incidence dual $G^*$, we can obtain information about walks between edges of $G$.

\begin{theorem}
If $G$ is a simple oriented hypergraph and $k$ a non-negative integer, then 
\begin{equation*}
(i,j)\text{-entry of }A_{G^{\ast }}^{k}=w^{\pm }(v_{i}^{\ast },v_{j}^{\ast
};k)=w^{\pm }(e_{i},e_{j};k).
\end{equation*}
\end{theorem}

The next result gives an explicit relationship between the adjacency matrix of the line graph of an oriented signed graph and the incidence dual.

\begin{theorem}\label{AdjLineGraphandDual} Let $\Sigma =(\Gamma ,sgn )$ be a simple signed graph. Let $G=(\Sigma,\tau)$ be an oriented signed graph.  Then \[A_{(\Lambda(\Gamma),\tau_{\Lambda})}=A_{G^{\ast }}.\]
\end{theorem}

\begin{proof}
By Proposition \ref{SGisaOHG}, $G$ may be treated as both a signed
graph and a 2-uniform simple oriented hypergraph. Therefore, $G$ has an
incidence orientation $\sigma$ defined by $\sigma(v,e)=\tau(v,w)$, for all $(v,e)\in \mathcal{I}$. 

The $(i,j)$-entry of $A_{G^{\ast }}$ is 
\begin{align*}
sgn_{e^*}(v_i^*,v_j^*) = sgn_v(e_i,e_j) = -\sigma(v,e_i)\sigma(v,e_j) = -\tau(v,e_i)\tau(v,e_j)&= -\tau_{\Lambda}(e_i,e_ie_j)\tau_{\Lambda}(e_j,e_ie_j)\\
&= sgn_{\Lambda(\tau)}(e_ie_j),
\end{align*}
which is the $(i,j)$-entry of $A_{(\Lambda (\Gamma ),\tau _{\Lambda })}$.
Since both $A_{(\Lambda (\Gamma ),\tau _{\Lambda })}$ and $A_{G^{\ast }}$
are $m\times m$ matrices, the proof is complete.
\end{proof}

%==========================================================
%==========================================================
\section{Incidence, degree, and Laplacian matrices}\label{IncLapSection}
%==========================================================
%==========================================================

Given a labeling $v_{1}$, $v_{2}$,\ldots , $v_{n}$ of the elements of $V$,
and $e_{1}$, $e_{2}$,\ldots , $e_{m}$ of the elements of $E$ of an oriented
hypergraph $G$, the \emph{incidence matrix} $\mathrm{H}_{G}=[\eta _{ij}]$ is the $n\times m$ matrix defined by 
\begin{equation}\label{etaelements}
\eta _{ij}=\sum_{k=1}^{\iota (v_{i},e_{j})}\sigma (v_{i},e_{j},k)\text{.}
\end{equation}
If $G$ is simple, then Equation \eqref{etaelements} is equivalent to
\begin{equation*}
\eta _{ij}=
\begin{cases} \sigma(v_{i},e_{j}) & \text{if }(v_{i},e_{j})\in \mathcal{I},\\
0 &\text{otherwise.}
\end{cases}
\end{equation*}

For signed and unsigned graphs, the line graph is the graphical
approximation of incidence duality. One advantage of incidence duality is
that it is an involution and provides a hypergraphic analog to transposition
of matrices.

\begin{theorem}
\label{T2}If $G$ is an oriented hypergraph, then $\mathrm{H}_{G}^T=%
\mathrm{H}_{G^{\ast }}$.
\end{theorem}

Another familiar graphic matrix that has a natural extension to oriented
hypergraphs is the degree matrix.  The \emph{degree matrix} of an oriented
hypergraph $G$ is $D_{G}=[d_{ij}]:=\text{diag}(\deg (v_{1}),\ldots ,\deg
(v_{n}))$.

The \emph{Laplacian matrix} is defined as $L_{G}:=D_{G}-A_{G}.$

\begin{theorem}
\label{T4}If $G$ is a simple oriented hypergraph, then $L_{G}=\mathrm{H}_{G}\mathrm{H}_{G}^T$. 
\end{theorem}

\begin{proof}
The $(i,j)$-entry of $\mathrm{H}_{G}\mathrm{H}_{G}^{T}$ corresponds to the $%
i^{\text{th}}$ row of $\mathrm{H}_{G}$, indexed by $v_{i}\in V$, multiplied
by the $j^{\text{th}}$ column of $\mathrm{H}_{G}^{T}$, indexed by $v_{j}\in V
$. Therefore, this entry is precisely 
\begin{equation*}
\sum_{e\in E}\eta _{v_{i}e}\eta _{v_{j}e}=\sum_{e\in E}\sigma
(v_{i},e)\sigma (v_{j},e).
\end{equation*}
If $i=j$, then the sum simplifies to 
\begin{equation*}
\sum_{e\in E}|\sigma (v_{i},e)|^{2}=\text{deg}(v_{i}),
\end{equation*}
since $|\sigma (v_{i},e)|=1$ if $v_{i}$ is incident to $e$. If $i\neq j$,
then the sum simplifies to 
\begin{equation*}
\sum_{e\in E}\sigma (v_{i},e)\sigma (v_{j},e)=\sum_{e\in
E}-sgn_{e}(v_{i},v_{j})=-a_{ij}.
\end{equation*}
Hence, $\mathrm{H}_{G}\mathrm{H}_{G}^{T}=D_{G}-A_{G}=L_{G}$.
\end{proof}

Vertex switching an oriented hypergraph $G$ can be described as matrix multiplication of the incidence matrix, and matrix conjugation of the Laplacian matrix.

\begin{proposition} Let $G$ be an oriented hypergraph.  If $\theta$ is a vertex switching function on $G$, then 
\begin{enumerate}
\item $\mathrm{H}_{G^{\theta}}=D(\theta)\mathrm{H}_{G}$, and
\item $L_{G^{\theta}}=D(\theta)^T L_{G} D(\theta)$.
\end{enumerate}
\end{proposition}

The following corollary is immediate from Theorems \ref{T2} and \ref%
{T4}.  This generalizes the classical relationship between the incidence, degree, adjacency and Laplacian matrices known for graphs and signed graphs.

\begin{corollary}
\label{T5}The following relations hold for a simple oriented hypergraph $G$.

\begin{enumerate}
\item $L_{G}=D_{G}-A_{G}=\mathrm{H}_{G}\mathrm{H}_{G}^T$,

\item $L_{G^{\ast }}=D_{G^{\ast
}}-A_{G^{\ast }}=\mathrm{H}_{G}^T\mathrm{H}_{G}.$
\end{enumerate}
\end{corollary}

If $G$ is a simple $k$-uniform oriented hypergraph, we can specialize the above as follows.
\begin{corollary}
\label{T6} If $G$ is a simple $k$-uniform oriented hypergraph, then%
\begin{equation*}
\mathrm{H}_{G}^T\mathrm{H}_{G}=kI-A_{G^{\ast }}.
\end{equation*}
\end{corollary}

\begin{proof}
Since every edge has size $k$, the incidence dual is $k$-regular.  Hence, $D_{G^{\ast }}=kI$.
\end{proof}

Now we have enough machinery to show that Corollaries \ref{T5} and \ref{T6} are generalizations of signed graphic relations, and can also be viewed as a generalzation of the classical relationship known for graphs (see \cite{MITTOSSG} for details).

\begin{corollary}
\label{T7} Let $\Sigma =(\Gamma ,sgn )$ be a simple signed graph. Let $G=(\Sigma,\tau)$ be an oriented signed graph.  Then
\begin{enumerate}
\item $\mathrm{H}_{G}\mathrm{H}_{G}^{T}=D_{G}-A_{G}=L_{G}$,
\item $\mathrm{H}_{G}^{T}\mathrm{H}_{G}=2I-A_{(\Lambda (\Gamma ),\tau _{\Lambda })}$.
\end{enumerate}
\end{corollary}
\begin{proof}
By Corollary \ref{T5}, (1) is immediate.  Since $G$ is 2-uniform, by Corollaries \ref{T5} and \ref{T6}, $\mathrm{H}_{G}^T\mathrm{H}_G = 2I-A_{G^*}$.  Also, by Theorem \ref{T6}, $A_{G^*}=A_{(\Lambda(\Gamma),\tau_{\Lambda})}$. The result follows.
\end{proof}

%==========================================================
%==========================================================
\section{Walk and weak walk matrices}\label{WWSection}
%==========================================================
%==========================================================

Here we study the underlying structure of the adjacency, incidence, and degree matrices to provide a unified combinatorial interpretation of the entries of $L$.

Let $A_{1},A_{2}\in \{V,E\}$ and $X_{(G,A_{1},A_{2},k)}=[x_{ij}]$ be the $A_{1}\times A_{2}$ matrix where $x_{ij}=w^{\pm }(a_{i},a_{j};k)$. \ The matrix $X_{(G,A_{1},A_{2},k)}$ is called a $k$\emph{-walk matrix} of an oriented hypergraph $G$.  If $A_{1}=A_{2}$, then we will assume that $k$ is a nonnegative integer.  If $A_{1}\neq A_{2}$, then we will assume that $k$ is a nonnegative half-integer.

\begin{lemma}
\label{W1}If $G$ is a simple oriented hypergraph, then $X_{(G,V,V,k)}=A^{k}$.
\end{lemma}

\begin{proof} The result is immediate by Theorem \ref{T1}.
\end{proof}

Walk matrices allow us to present the incidence matrix as a measure of
$1/2$-walks.

\begin{lemma}\label{HalfWalkIncidenceMatrix}
\label{W3}If $G$ is an oriented hypergraph, then $X_{(G,V,E,1/2)}=\mathrm{H}%
_{G}$.
\end{lemma}

\begin{proof}
The $(v,e)$-entry of $X_{(G,V,E,1/2)}$ is $x_{ve}=w^{\pm }(v,e;1/2)=\sigma(v,e)=\eta_{ve}$.\qedhere
\end{proof}

The incidence dual again provides new walk information on $G$.
\begin{lemma}
\label{W2}The following duality relationships hold for $k$-walk matrices.

\begin{enumerate}
\item $X_{(G,V,V,k)}=X_{(G^{\ast },E^{\ast },E^{\ast },k)}$,

\item $X_{(G,E,E,k)}=X_{(G^{\ast },V^{\ast },V^{\ast },k)}$,

\item $X_{(G,V,E,k)}^{T}=X_{(G,E,V,k)}=X_{(G^{\ast },V^{\ast },E^{\ast },k)}$%
.
\end{enumerate}
\end{lemma}

The following corollary translates the Laplacian matrix in terms of $1/2$-walks.  It says that if $G$ is an oriented hypergraph, then the $(i,j)$-entry of $L_G$ represents the number of half walks from $v_i$ to some edge $e$ multiplied by the number of half walks from edge $e$ to $v_j$.

\begin{corollary} If $G$ is an oriented hypergraph, then $L_G=X_{(G,V,E,1/2)}X_{(G,E,V,1/2)}$.
\end{corollary}
\begin{proof} By Theorem \ref{T4} and Lemma \ref{HalfWalkIncidenceMatrix}, 
\[ L_G =\mathrm{H}_{G}(\mathrm{H}_{G})^T=X_{(G,V,E,1/2)}(X_{(G,V,E,1/2)})^T=X_{(G,V,E,1/2)}X_{(G,E,V,1/2)}.\qedhere\]
\end{proof}

A \emph{weak walk} is a walk in which the condition $i_{2h-1}\neq i_{2h}$ is
removed and you are allowed to immediately return along the same incidence
at any point.  The sign of a weak walk is given by Equation \eqref{SignOfAWalk}.

Paralleling the notation $w(a_{i},a_{j};k)$ for the number of walks of
length $k$ from $a_{i}$ to $a_{j}$, we let $\widetilde{w}(a_{i},a_{j};k)$
denote the number of weak walks of length $k$ between $a_{i}$ and $a_{j}$,
and define $\widetilde{w}^{+}(a_{i},a_{j};k)$, $\widetilde{w}%
^{-}(a_{i},a_{j};k)$, $\widetilde{w}^{\pm }(a_{i},a_{j};k)$ for weak walks
analogously. A weak walk of length $1$ of the form $v,i,e,i,v$ is called a 
\emph{backstep}.

\begin{theorem}
\label{W4}If $G$ is an oriented hypergraph, then the $(i,j)$-entry of $D_{G}$
is $\widetilde{w}(v_{i},v_{j};1)-w(v_{i},v_{j};1)$, the number of backsteps
at $v_{i}$.
\end{theorem}

\begin{proof}
The weak walks of length $1$ from a fixed vertex $v_{i}$ must contain two
incidences. \ These naturally partition into those walks that return along
the same incidence, or backsteps, and those that form adjacencies. \ The
value $w(v_{i},v_{j};1)$ is precisely those that form adjacencies, so $%
\widetilde{w}(v_{i},v_{j};1)-w(v_{i},v_{j};1)$ counts the incidences at vertex $v_i$.
\end{proof}

The next theorem provides an alternative way to calculate the entries of the Laplacian matrix of an oriented hypergraph.

\begin{theorem}
\label{W6}If $G$ is an oriented hypergraph, then the $(i,j)$-entry of $L_{G}$
is 
\begin{equation*}
\ell _{ij}=\widetilde{w}(v_{i},v_{j};1)-2w^{+}(v_{i},v_{j};1).
\end{equation*}
\end{theorem}

\begin{proof}
The $(i,j)$-entry of $L(G)$ is
\begin{align*}
\ell _{ij}=d_{ij}-a_{ij} &=[\widetilde{w}(v_{i},v_{j};1)-w(v_{i},v_{j};1)]-w^{\pm
}(v_{i},v_{j};1)\quad \text{(by Lemma }\ref{W1}\text{ and Theorem }\ref{W4}\text{.)} \\
%& =\widetilde{w}%
%(v_{i},v_{j};1)-[w^{+}(v_{i},v_{j};1)+w^{-}(v_{i},v_{j};1)]-w^{\pm
%}(v_{i},v_{j};1) \\
%& =\widetilde{w}(v_{i},v_{j};1)-w^{+}(v_{i},v_{j};1)-w^{+}(v_{i},v_{j};1) \\
& =\widetilde{w}(v_{i},v_{j};1)-2w^{+}(v_{i},v_{j};1).\qedhere
\end{align*}
\end{proof}

\begin{lemma}
\label{W7}The sign of every backstep is negative.
\end{lemma}

The \emph{weak }$\emph{k}$\emph{-walk matrix} $%
W_{(G,A_{1},A_{2},k)}=[w_{ij}] $ is defined similar to the $k$-walk matrix
except $w_{ij}=\widetilde{w}^{\pm }(a_{i},a_{j};k)$.

Theorem \ref{W6} allowes us to calculate the entries of the Laplacian matrix of an oriented hypergraph via walk information.  Here we present a unified combinatorial interpretation of the Laplacian matrix entries as weak walks of length 1.

\begin{theorem}
\label{W8}If $G$ is a simple oriented hypergraph, then $L_G=-W_{(G,V,V,1)}$.
\end{theorem}

\begin{proof}
If $i=j$, then every weak $1$-walk is a backstep since $G$ is simple. Thus, 
\begin{align*}
w_{ii} =\widetilde{w}^{\pm }(v_{i},v_{i};1) =\widetilde{w}^{+}(v_{i},v_{i};1)-\widetilde{w}^{-}(v_{i},v_{i};1) &=0-\widetilde{w}^{-}(v_{i},v_{i};1)\quad \text{(by Lemma }\ref{W7}) \\
&=-\deg (v_{i}).
\end{align*}
If $i\neq j$, then there are no backsteps and each weak walk is a walk. Thus,
\[ w_{ij} = \widetilde{w}^{\pm }(v_{i},v_{j};1) = w^{\pm }(v_{i},v_{j};1) = a_{ij}. \] 
Hence, $W_{(G,V,V,1)}=-D_G+A_G=-L_G$.
\end{proof}

Theorems \ref{W6} and \ref{W8} are summarized here.

\begin{corollary}If $G$ is a simple oriented hypergraph, then the $(i,j)$-entry of $L_{G}$ is 
\begin{equation*}
\ell _{ij}=\widetilde{w}(v_{i},v_{j};1)-2w^{+}(v_{i},v_{j};1)=-\widetilde{w}%
^{\pm }(v_{i},v_{j};1).
\end{equation*}
\end{corollary}

%%%%%%%%%%%%%%%%%%%%%%%%%%%%%%%%%%%%
%%%%%%%%%%%%%%%%%%%%%%%%%%%%%%%%%%%%
%\nocite{*} 
\bibliographystyle{amsplain2}
\bibliography{mybib}

\end{document}